\documentclass[11pt,a4paper,reqno, draft]{article}

\usepackage{template}
\usepackage{bbm}
\usepackage{marvosym}

\newcommand{\mc}[1]{\mathcal{#1}}
\newcommand{\bb}[1]{\mathbb{#1}}

\newcommand{\E}{\mathbb{E}}
\newcommand{\N}{\mathbb{N}}

\newcommand{\F}{\mathcal{F}}

\begin{document}

\title{A proof of Sanov’s Theorem via discretizations}

\date{\today}
\author{Rangel Baldasso 
  \thanks{ mail: r.baldasso@math.leidenuniv.nl, Mathematical Institute, Leiden University, P.O. Box 9512, 2300 RA Leiden, The Netherlands.}
  \and
  Roberto I. Oliveira 
  \thanks{rimfo@impa.br, IMPA, Estrada Dona Castorina 110, 22460-320 Rio de Janeiro, RJ - Brazil.}
  \and
  Alan Pereira 
  \thanks{alan.pereira@im.ufal.br, Instituto de Matem\'{a}tica, Universidade Federal de Alagoas, Rua Lorival de Melo Mota s/n, 57072970 Macei\'{o}, AL - Brazil.}
  \and
  Guilherme Reis 
  \thanks{\Letter guilherme.reis@tum.de, Fakult\"at f\"ur Mathematik, Technische Universit\"at M\"unchen, Boltzmannstra{\ss}e 3, 85748 Garching bei M\"unchen, Germany.}}
\maketitle

\begin{abstract}
We present an alternative proof of Sanov’s theorem for Polish spaces in the weak topology that follows via discretization arguments. We combine the simpler version of Sanov's Theorem for discrete finite spaces and well chosen finite discretizations of the Polish space. The main tool in our proof is an explicit control on the rate of convergence for the approximated measures.

\noindent
\emph{Keywords and phrases:} large deviations; Sanov's Theorem

\noindent
MSC 2010: 60F10
\end{abstract}

\section{Introduction}
~
\par Sanov's Theorem is a well know result in the theory of large deviations principles. It provides the large deviations profile of the empirical measure of a sequence of i.i.d. random variables and characterizes its rate function as the relative entropy. This short note provides an alternative proof of this fact, by exploring the metric structure of the weak topology with the variational formulation of the relative entropy.

Formally, let $(M,\dist)$ be a Polish space and let $\big(X_n \big)_{n \in \N}$ be a sequence of independent $M$-valued random elements identically distributed according to $\mu \in \mc{P}(M)$,  where $\mc{P}(M)$ is the set of Borel probability measures on $M$. We denote by $\delta_x$ the probability measure degenerate at $x \in M$, and define the \textit{empirical measure} of $X_1, \dots , X_n$  by
\begin{equation} \label{eq:DefLn}
L_n := \frac{1}{n}\sum_{i=1}^n\delta_{X_i}.
\end{equation}
Also, given $\upsilon, \mu \in \mc{P}(M)$, the \textit{relative entropy} between $\upsilon$ and $\mu$ is defined as
\begin{equation} \label{eq:variational}
H(\upsilon|\mu):=\sup \left\{\int f \d \upsilon -\log \int e^f \d \mu;\,f \text{ is measurable and bounded} \right\}.
\end{equation}

Sanov's Theorem is given by the following statement.
\begin{theorem}[Sanov] \label{theoem:sanov}
Let $\big(X_n \big)_{n \in \N}$ be a sequence of i.i.d.\ random variables taking values in a Polish space $(M,\dist)$ with distribution $\mu \in \mc{P}(M)$. The sequence of empirical measures $\big(L_n\big)_{n \in \N}$ of $\big(X_{n}\big)_{n \in \N}$ (defined in Equation~\eqref{eq:DefLn}) satisfies a large deviations principle on the space $\mathcal{P}(M)$ with rate function $H( \,\cdot\, |\mu)$.
\end{theorem}

When the space $M$ is finite, the theorem above is proved in an elementary and elegant way (see Den Hollander~\cite[Theorem II.2]{denhollander}, and Dembo and Zeitouni~\cite[Theorem 2.1.10]{dembo_zeitouni}). In this work, we prove the theorem for general Polish metric spaces by extending this elementary proof via sequences of discretizations of the space.
We split the set $M$ in a finite number of subsets which belong to one of two distinct categories. The well-behaved sets are the ones with small diameter, while the badly behaved sets will have small $\mu$-measure. We remark though that when the space $M$ is compact, no badly behaved sets are necessary. These partitions define natural projections on the space and allow us to approximate the sequence $\big(X_{n} \big)_{n \in \N}$ by variables in the discretized spaces, and by consequence provide approximations for its empirical measures. The main technical observation is that the discretized relative entropy converges to the relative entropy~\eqref{eq:variational} as we take thinner partitions (see Lemma~\ref{le:supgoodpart}) and the relative entropy is well approximated in balls (see Lemma~\ref{lemma:supinf}).

Some ideas used to prove Lemma~\ref{lemma:supinf} are roughly inspired by the proof of the upper bound in Csisz\'ar~\cite{csiszar}. His work presents a proof of Sanov's Theorem for the $\uptau$-topology, a stronger topology than that of weak convergence, with an approach that differs greatly from more classical ones that can be found, for example, in~\cite[Theorem 6.2.10]{dembo_zeitouni}.

There are two proofs of Sanov's Theorem in~\cite{dembo_zeitouni}, one by means of Cram\'er's Theorem for Polish spaces and the other following a projective limit approach. Although we strongly use the metric structure of the space, our proof does not require profound knowledge of large deviations theory or general topology.

\bigskip

\noindent\textbf{Organization of the paper.} In the next section we collect some preliminary notation and results that are used during the text. Section~\ref{sec:discretization} introduces the discretization considered here. Section~\ref{sec:theorem} contains the statement of the main lemmas used in the proof. We also show how Sanov's Theorem is proved in Section~\ref{sec:theorem}. Sections~\ref{sec:proof:le:supgoodpart} and~\ref{section:proof_supinf} contain the proofs of Lemmas~\ref{le:supgoodpart} and \ref{lemma:supinf}, respectively.

\bigskip

\noindent\textbf{Acknowledgments.}
RB is supported by the Mathematical Institute of Leiden University for support.
RIO counted on the support of CNPq, Brazil via a {\em Bolsa de Produtividade em Pesquisa} (304475/2019-0) and a {\em Universal} grant (432310/2018-5).
GR was partially supported by  a Capes/PNPD fellowship 888887.313738/2019-00 while he was a post doc at Federal University of Bahia (UFBA).

\noindent\textbf{Data Availability Statement .}
Data sharing is not applicable to this article as no datasets were generated or analysed during the current study.

\section{Preliminaries}
~
\par In this section we review some basic concepts. We provide the definition of large deviations principle and weak topology, and collect some properties of the relative entropy.

\begin{definition}[Large deviation principle]\label{def:LDP} A sequence $\big( \bb{P}_n \big)_{n \in \N}$ of probabilities over a metric space $(\mathfrak{X}, \d_{\mathfrak{X}})$ satisfies a large deviation principle with rate function $I$ if

\begin{enumerate}
\item (Lower bound) For any open set $\mc{O} \subset \mathfrak{X}$,
\begin{equation}
\liminf_{n \to \infty}\frac{1}{n} \log  \bb{P}_n (\mc{O}) \geq -\inf_{x \in \mc{O}} I(x);
\end{equation}

\item (Upper bound) For any closed set $\mc{C} \subset \mathfrak{X}$,
\begin{equation}
\limsup_{n \to \infty}\frac{1}{n} \log  \bb{P}_n(\mc{C}) \leq -\inf_{x \in \mc{C}} I(x).
\end{equation}
\end{enumerate}
\end{definition}

The weak topology on $\mc{P}(M)$ is defined as the topology generated by the functionals
\begin{equation}
\upsilon \mapsto \int \varphi \, \d \upsilon,
\end{equation}
where $\varphi \in C_b(M)$ is a continuous bounded function.

A metric compatible with the weak topology is the bounded Lipschitz metric
\begin{equation}
\d_{\rm BL}(\mu, \upsilon) = \sup \left \{ \left| \int f \, \d \upsilon - \int f \, \d \mu \right|: f \in BL(M) \right\},
\end{equation}
where $BL(M)$ is the class of $1$-Lipschitz functions $f : M \to \bb{R}$ bounded by one.

For the next lemma, let $x \wedge y$ denote the minimum between $x$ and $y$.
\begin{lemma}
Let $(X, Y)$ be a coupling of two distributions $\mu$ and $\upsilon$. Then
\begin{equation}
\d_{\rm BL}(\mu,\upsilon)\leq \bb{E}(\d(X,Y) \wedge 2).
\end{equation}
\end{lemma}

\begin{proof}
Let $x, y \in M$ and notice that, for each $f \in BL(M)$,
\begin{equation}
|f(x) -f(y)| \leq \d(x,y) \wedge 2,
\end{equation}
since $f$ is 1-Lipschitz and bounded by one. The proof is now complete by noting that
\begin{equation}
\left| \int f \, \d \upsilon - \int f \, \d \mu \right| = \big| \bb{E} \big(f(X)-f(Y) \big) \big| \leq \bb{E}(\d(X,Y) \wedge 2).
\end{equation}
\end{proof}

\bigskip

Equation~\eqref{eq:variational} is called the variational formulation of entropy, and it readily implies the so-called \textit{entropy inequality}
\begin{equation}\label{eq:entropy_inequality}
\int f \d \upsilon \leq H(\upsilon|\mu) + \log \int e^f \d \mu,
\end{equation}
for any measurable bounded function $f$. We will also make use of the \textit{integral formulation} of the relative entropy, provided in the next lemma. This formulation will be the key result used to approximate relative entropies in the discrete case to the general case.
\begin{lemma}\label{lemma:entropy}
The variational formula of the relative entropy~\eqref{eq:variational} is equivalent to the following integral formulation of the entropy
\begin{equation}
H(\upsilon|\mu)=\begin{cases}
\int \frac{\d\upsilon}{\d\mu}\log \frac{\d\upsilon}{\d\mu} \d \mu, & \text{ if } \upsilon \ll \mu, \\
+\infty, & \text{ otherwise}.
\end{cases}
\end{equation}
\end{lemma}

We refrain from presenting the proof of the lemma above, and refer the reader to~\cite[Theorem 5.2.1]{gray}.

\section{Discretization}\label{sec:discretization}
~
\par In this section we present the discretization procedure used for the space $M$ and related constructions for measures and random variables.

We start by discretizing the space. Let $\mu \in \mc{P}(M)$ and recall that, since $M$ is a Polish space, there exists, for each $m \in \bb{N}$, a compact set $K_m$ with
\begin{equation}\label{eq:choice_K}
\mu(K_{m}^{\complement}) \leq \frac{e^{-m^{2}-1}}{m}.
\end{equation}

The support of the measure $\mu$ is contained in the closure of the union of the compacts $K_{m}$. Notice that the collection of probability measures supported on the closure of $\cup_{m=1}^{\infty} K_{m}$ forms a closed subset of $\mc{P}(M)$, and thus it is enough to prove a large deviations principle for this subspace (see~\cite[Lemma 4.1.5]{dembo_zeitouni}). We assume from now on that
\begin{equation}
M = \overline{\bigcup_{m=1}^{\infty} K_{m}}.
\end{equation}

Given a sequence of partitions $\big( \mc{A}_{m} \big)_{m \in \bb{N}}$, let $\mathcal{F}_{m}$ and $\mathcal{F}_{\infty}$ denote the $\sigma$-algebras generated by $\mathcal{A}_{m}$ and by the union $\cup_{m=1}^{\infty}\mathcal{A}_{m}$, respectively. We write $\mathcal{B}(M)$ for the Borel $\sigma$-algebra in $M$.
\begin{lemma}\label{lemma:good_partition}
There exists a sequence of nested partitions $\big( \mc{A}_{m} \big)_{m \in \bb{N}}$ such that $\mc{A}_{m} = \{ A_{m, 1}, \dots, A_{m, \ell_{m}} \}$ and
\begin{itemize}
\item $\diam(A_{m, i})<\frac{1}{m}$, if $i=1,\dots, \tilde{\ell}_{m}$, for some $\tilde{\ell}_{m} \leq \ell_{m}$.

\item $K_{m}^{\complement} =  \bigcup_{i=\tilde{\ell}_{m}+1}^{\ell_{m}} A_{m, i}$.

\item $\mathcal{F}_{\infty} = \mathcal{B}(M)$.
\end{itemize}
\end{lemma}

\begin{proof}
Notice that if we can construct partitions $\mc{A}_{m}$ for each $m$ that satisfy the three first requirements of the lemma without requiring them to be nested, it is possible to take refinements in order to obtain a nested sequence.

Recall the definition the compact set $K_{m}$ in~\eqref{eq:choice_K}. By compactness, it is possible to partition $K_{m}$ into subsets $\{ C_{m, 1}, \dots, C_{m, \bar{\ell}_m} \}$ of diameter at most $\frac{1}{m}$, so that $\mc{C}_m=\{K_{m}^{\complement}, C_{m, 1}, \dots, C_{m, \bar{\ell}_m}\}$ defines a partition of $M$.

Consider an enumeration $\big(B^{i}\big)_{i \in \bb{N}}$ of balls of rational radius and centered in a countable dense subset of $M$. We now define the partition $\mc{A}_{m}$ via the intersections of sets in $\mc{C}_{m}$ with $B^{m}$ and its complement. We write
\begin{equation}
\mc{A}_{m} = \{ A_{m, 1}, \dots, A_{m, \ell_{m}} \},
\end{equation}
where $A_{m, i}$, $i \leq \tilde{\ell}_{m}$, denotes the sets contained in $K_{m}$ and  $A_{m, i}$, $\tilde{\ell}_{m}+1 \leq i \leq \ell_{m}$, indicates the sets contained in $K_{m}^{\complement}$.

Notice that the first two statements about the partition $\mc{A}_{m}$ are immediately verified. To check the last claim, notice that $B^{i} \in \mathcal{F}_{i}$, and thus $B^{i} \in \mc{F}_{\infty}$, for all $i \in \bb{N}$, which implies $\mathcal{F}_{\infty} = \mathcal{B}(M)$ and concludes the proof.
\end{proof}

We select a subset $M_m:=\{a_{m, 1}, \dots, a_{m, \ell_{m}} \} \subset M$ such that $a_{m, i} \in A_{m, i}$ for $i=1, \dots,\ell_m$ and turn $(\mc{A}_{m},M_{m})$ into a tagged partition. We will furthermore assume that $M_m \subset M_{m+1}$.

For each $m \in \bb{N}$, the tagged partition $(\mc{A}_{m}, M_{m})$ defines a natural projection $\pi^{m}: M \to M_{m}$ via
\begin{equation}
\pi^m(x)=a_{m, i}, \text{ if } x \in A_{m, i}.
\end{equation}
This allows us to define, for any measure $\upsilon \in \mc{P}(M)$, its discretized version $\upsilon^{m} \in \mc{P}(M)$ as the probability measure supported in $M_m$ given by the pushforward of $\upsilon$ via the map $\pi^{m}$, i.e.
\begin{equation}
\upsilon^{m}(a_{m, i})=\upsilon \big( (\pi^{m})^{-1}(a_{m, i}) \big) = \upsilon(A_{m, i}), \text{ for }i=1,\dots,\ell_m.
\end{equation}
Random elements are also discretized with the aid of the projection maps $\pi^{m}$. If $\big( X_i \big)_{i \in \bb{N}}$ is an i.i.d.\ sequence of random elements with distribution $\mu \in \mc{P}(M)$, then $X_{i}^{m}=\pi^m(X_i)$ yields an i.i.d.\ sequence of random elements distributed according to $\mu^{m}$. 

The empirical measure for the discretized elements is given by
\begin{equation} \label{eq:DefLnk}
L_n^{m}:=\frac{1}{n}\sum_{i=1}^n\delta_{X_i^{m}}.
\end{equation}
Since, for each $m \in \bb{N}$, the elements $X_{i}^{m}$ take values on the finite space $M_{m}$, we know that the sequence of empirical measures $\big( L_n^m \big)_{n \in \bb{N}}$ satisfies a large deviations principle on the space $\mc{P}(M_{m})$ with rate function $H( \,\cdot\, |\mu^m)$. Via~\cite[Lemma 4.1.5]{dembo_zeitouni}, we can extend these large deviation principles to the whole space $\mc{P}(M)$ with rate function also given by $H( \,\cdot\, |\mu^m)$ (note that $H( \upsilon |\mu^m)$ is infinite if $\upsilon \notin \mc{P}(M_{m})$).

Lemma~\ref{lemma:entropy} yields the following expression for the rate function $H( \upsilon|\mu^m)$, when $\upsilon \in \mc{P}(M_{m})$:
\begin{equation}
H(\upsilon|\mu^m) = \sum_{a \in M_m}\upsilon(a)\log \frac{\upsilon(a)}{\mu^m(a)}.
\end{equation}

\section{Proof of Theorem \ref{theoem:sanov}}\label{sec:theorem}
~
\par In this section we present our approach to the proof of Sanov's Theorem. Our goal is to deduce that the empirical measures $L_{n}$ given by~\eqref{eq:DefLn} satisfy a large deviations principle from the information that the sequences $\big( L_n^m \big)_{n \in \bb{N}}$ satisfy large deviations principles, for all $m \in \bb{N}$. Since the rate function given by Sanov's Theorem (Theorem~\ref{theoem:sanov}) is the relative entropy, the following two lemmas that relate the entropies in discrete and Polish spaces are the central pieces in our proof.
\begin{lemma}\label{le:supgoodpart} For any two probability measures $\mu, \upsilon \in \mc{P}(M)$,
\begin{equation}\label{eq:limit_entropy}
\sup_m H(\upsilon^m|\mu^m)=\lim_{m} H(\upsilon^m|\mu^m)=H(\upsilon|\mu).
\end{equation}
Furthermore, if $\sup_m H(\upsilon^{m}| \mu^m)<\infty$ then there exists a positive constant $c \geq 0$ such that
\begin{equation}
\d_{\rm BL}(\upsilon^m,\upsilon)\leq \frac{c}{m}.
\end{equation}
\end{lemma}

\begin{remark}\label{remark:rate_function}
Notice that the lemma above provides an alternative expression for the rate function in Sanov's Theorem, depending on the sequence of partitions chosen. Even more is true: the supremum in~\eqref{eq:limit_entropy} can be taken over all finite partitions of $M$. In fact, simply notice that, if $(\mathcal{A}, \mathcal{M})$ is a dotted partition of $M$ with projection map $\pi: M \to \mathcal{M}$, then any function $f: \mathcal{M} \to \bb{R}$ can be extended to $M$ via $\tilde{f}=f \circ \pi$. To conclude, apply the variational formulation of relative entropy to obtain $H(\upsilon^{\mathcal{A}}|\mu^\mathcal{A}) \leq H(\upsilon|\mu)$.
\end{remark}

\begin{lemma}\label{lemma:supinf} For any $\mu, \upsilon \in \mc{P}(M)$ and $m_{0} \in \bb{N}$,
\begin{equation}
\sup_{m \geq m_{0}} \inf_{\sigma \in \bar{B}_{\frac{1}{\sqrt{m}}}(\upsilon)}H(\sigma|\mu^m)=H(\upsilon|\mu).
\end{equation}
\end{lemma}

We prove Lemma~\ref{le:supgoodpart} in Section~\ref{sec:proof:le:supgoodpart} and Lemma~\ref{lemma:supinf} in Section~\ref{section:proof_supinf}.

The next lemma is a result of exponential equivalence.
\begin{lemma}\label{prop:boundLkLnk}
Let $L_n$ and $L_n^m$ as defined in Equations~\eqref{eq:DefLn} and~\eqref{eq:DefLnk}, respectively. Then
\begin{equation}
\bb{P}\bigg( \d_{\rm BL}(L_n,L_n^m) > \frac{3}{m}\bigg) \leq \exp\big(-mn\big).
\end{equation}
\end{lemma}

\begin{proof}

Observe that if $X_i \in A_{m, j}$ for some $j=1,\dots,\tilde{\ell}_m$ then $d(X_i,X_i^{m})\leq 1/m$. In particular,
\begin{equation}
\d_{BL}(L_n,L_n^m)\leq \frac{1}{n}\sum_{i=1}^n \d(X_i,X_i^{m}) \wedge 2 \leq \frac{1}{m}+\frac{2}{n}\sum_{i=1}^n1_{\{X_i \in K_m^{\complement}\}}.
\end{equation}
This implies
\begin{equation}
\bb{P}\bigg( \d_{BL}(L_n,L_n^m)>\frac{3}{m}\bigg)\leq \bb{P}\bigg(\frac{1}{n}\sum_{i=1}^n 1_{\{X_i \in K_m^{\complement}\}}>\frac{1}{m}\bigg)
\end{equation}
In order to bound the last probability, we use union bound and independence to obtain
\begin{equation}
\begin{split}
\bb{P}\bigg(\frac{1}{n}\sum_{i=1}^n 1_{\{X_i \in K_m^{\complement}\}}>\frac{1}{m}\bigg) & \leq \sum_{A \subset [n]: |A|=\frac{n}{m}} \bb{P} \big(X_i \in K_m^{\complement}, \text{ for all } i \in A \big) \\
& \leq \binom{n}{\frac{n}{m}}\Big( \frac{e^{-m^{2}-1}}{m} \Big)^{\frac{n}{m}} \\
& \leq \bigg(em\frac{e^{-m^{2}-1}}{m} \bigg)^{\frac{n}{m}} \\
& \leq \exp\big(-mn\big),
\end{split}
\end{equation}
concluding the proof.
\end{proof}

We are now ready to work on the proof of Theorem~\ref{theoem:sanov}. It is proved in~\cite[Lemma 6.2.6]{dembo_zeitouni} that the sequence $\big( L_n \big)_{n \in \bb{N}}$ is exponentially tight. In particular, there exists a subsequence $\big( L_{n_k} \big)_{n_{k}}$ that satisfies a large deviations principle with rate function $I$.

From now on, we drop the subscript $k$ in $n_{k}$. Notice that
\begin{equation} \label{eq:defI}
-I(\upsilon) = \lim_{\varepsilon\to 0} \lim_{n\to\infty} \frac{1}{n}\log \bb{P}(L_n\in B_\varepsilon(\upsilon)).
\end{equation}
Even though the rate function $I$ might depend on the subsequence, our goal is to prove that this is not the case. In fact, we prove that $I(\, \cdot \,)=H( \, \cdot \,|\mu)$. In Proposition~\ref{prop:HgeqI}, we prove that $H( \, \cdot \,|\mu) \geq I(\, \cdot \,)$, while the opposite inequality is established in Proposition~\ref{prop:HleqI}. This concludes the proof of Theorem~\ref{theoem:sanov}, since any possible subsequence $\big( L_{n_{k}} \big)_{n_{k}}$ that satisfies a large deviations principle does so with the same rate function $H( \,\cdot\, | \mu)$, which implies that the whole sequence also satisfies a large deviations principle.

\begin{proposition} \label{prop:HgeqI}
The function $I$ in Equation~\eqref{eq:defI} satisfies $H(\, \cdot \,|\mu)\geq I(\, \cdot \,)$.
\end{proposition}

\begin{proof}
Fix $\upsilon \in \mc{P}(M)$ and notice we can assume that $H(\upsilon | \mu)$ is finite since the statement is trivially verified if otherwise.

Due to Lemma~\ref{le:supgoodpart}, we have
\begin{equation}
\d_{\rm BL}(\upsilon,\upsilon^m) \leq \frac{c}{m},
\end{equation}
for some positive constant $c>0$. In particular, this implies
\begin{equation}\label{eq:first_bound}
\bb{P}( L^m_n \in B_\varepsilon(\upsilon^m))\leq \bb{P}\Big( L_n \in \bar{B}_{\varepsilon+\frac{3+c}{m}}(\upsilon) \Big)+\bb{P} \Big( \d_{\rm BL}(L_n,L_n^m)> \frac{3}{m} \Big),
\end{equation}
which yields
\begin{equation}\label{eq:first_estiamte}
\begin{split}
\frac{1}{n} & \log \bb{P}(L^m_n \in B_\varepsilon(\upsilon^m)) \leq \frac{\log 2}{n} \\
& \qquad \qquad + \max \left\{\frac{1}{n}\log \bb{P}\Big( L_n \in \bar{B}_{\varepsilon+\frac{3+c}{m}}(\upsilon) \Big), \frac{1}{n}\log \bb{P} \Big( \d_{\rm BL}(L_n,L_n^m)> \frac{3}{m} \Big) \right\}.
\end{split}
\end{equation}

Lemma~\ref{prop:boundLkLnk} gives
\begin{equation}
\frac{1}{n}\log \bb{P} \Big( \d_{\rm BL}(L_n,L_n^m)> \frac{3}{ m} \Big) \leq -m,
\end{equation}
and, by taking the limit as $n$ grows in~\eqref{eq:first_estiamte},
\begin{equation}
-\inf_{\sigma \in B_{\varepsilon}(\upsilon^m)} H(\sigma|\mu^{m}) \leq \max \bigg\{-\inf_{\sigma \in \bar{B}_{\varepsilon+\frac{3+c}{m}}(\upsilon)}I(\sigma),-m \bigg\}.
\end{equation}
We now take the limit as $\varepsilon$ goes to zero to obtain
\begin{equation}
-H(\upsilon^m | \mu^m) \leq \max \bigg\{- \inf_{\sigma \in \bar{B}_{\frac{3+c}{m}}(\upsilon)} I (\sigma),-m \bigg\},
\end{equation}
which readily implies, via  Lemma~\ref{le:supgoodpart},
\begin{equation}\label{eq:hgI}
H(\upsilon | \mu) = \sup_{\bar m}H(\upsilon^{\bar m}|\mu^{\bar m}) \geq \min \bigg\{ \inf_{\sigma \in \bar{B}_{\frac{3+c}{m}}(\upsilon)} I(\sigma), m \bigg\},
\end{equation}
for all $m \in \bb{N}$.

Since the function $I$ is lower semicontinuous,
\begin{equation}
\lim_{m \to \infty} \inf_{\sigma \in \bar{B}_{\frac{3+c}{m}}(\upsilon)}I(\sigma)=I(\upsilon).
\end{equation}
In particular,
\begin{equation}
H(\upsilon | \mu) \geq I(\upsilon),
\end{equation}
concluding the proof.
\end{proof}

\begin{proposition} \label{prop:HleqI}
We have $H( \,\cdot\, |\mu)\leq I(\, \cdot \,)$.
\end{proposition}

\begin{proof}
Fix $\upsilon \in \mc{P}(M)$ and observe once again that 
\begin{equation*}
\begin{split}
\frac{1}{n} & \log \bb{P}\big( L_n \in B_\varepsilon(\upsilon) \big) \leq  \frac{\log 2}{n} \\
& \qquad \qquad + \max \bigg\{\frac{1}{n} \log \bb{P} \big( L_n^m \in \bar{B}_{\varepsilon+\frac{1}{\sqrt{m}}}(\upsilon) \big), \frac{1}{n} \log \bb{P} \big( \d_{\rm BL} (L_n,L_n^m)>\tfrac{1}{\sqrt{m}} \big) \bigg\}.
\end{split}
\end{equation*}

Taking $n \to \infty$ and $\varepsilon\to 0$, we obtain with the aid of Lemma~\ref{prop:boundLkLnk}
\begin{equation}\label{eq:bound_I}
-I(\upsilon) \leq \max \bigg\{-\inf_{\sigma \in \bar
{B}_{\frac{1}{\sqrt{m}}}(\upsilon)}H(\sigma|\mu^m),\,-m \bigg\}.
\end{equation}

Let us now split the discussion in whether $H(\upsilon|\mu)$ is finite or not. Assume first that this relative entropy is infinite and notice that Lemma~\ref{lemma:supinf} implies
\begin{equation}
\sup_{m}\inf_{\sigma \in \bar
{B}_{\frac{1}{\sqrt{m}}}(\upsilon)}H(\sigma|\mu^m) = \infty,
\end{equation}
which readily implies $I(\upsilon)=\infty$, when combined with~\eqref{eq:bound_I}.

If on the other hand we have $H(\upsilon|\mu)< \infty$, we combine~\eqref{eq:bound_I} and Lemma~\ref{lemma:supinf} with $m_{0} \geq H(\upsilon|\mu)$ to obtain
\begin{equation}
I(\upsilon) \geq \min \bigg\{ \inf_{\sigma \in \bar{B}_{\frac{1}{\sqrt{m}}}(\upsilon)} H(\sigma|\mu^m), m \bigg\} \geq \inf_{\sigma \in \bar{B}_{\frac{1}{\sqrt{m}}}(\upsilon)} H(\sigma|\mu^m),
\end{equation}
for every $m \geq m_{0}$. Taking the supremum in $m$ concludes the proof.
\end{proof}

\section{Proof of Lemma~\ref{le:supgoodpart}}\label{sec:proof:le:supgoodpart}
~
\par In this section we prove Lemma~\ref{le:supgoodpart}. We start with the following preliminary lemma, which in particular implies the second part of Lemma~\ref{le:supgoodpart}. We prove the first part afterwards.
\begin{lemma}\label{lemma:convergenceNukDiv}
If $\sigma \in \mc{P}(M)$ is such that $H(\sigma|\mu) \leq \alpha$, then, for any $\theta>0$, we have
\begin{equation}\label{eq:distance_estimate}
\d_{\rm BL}(\sigma^m,\sigma) \leq \frac{1}{m}+2\frac{\alpha}{\theta}+ 2\frac{e^{-m^{2}-1+\theta}}{m\theta}.
\end{equation}
In particular,
\begin{equation}
\d_{\rm BL}(\sigma^m,\sigma) \leq \frac{3+2\alpha}{m}, \text{ for all } m \in \bb{N}.
\end{equation}
\end{lemma}

\begin{proof}
Consider $X \sim \sigma$ and notice that $X^{m} = \pi^{m}(X)$ has distribution $\sigma^{m}$. Therefore,
\begin{equation}
\d_{\rm BL}(\sigma^m,\sigma)\leq \bb{E}(\d(X^m,X)\wedge 2).
\end{equation}
Splitting on whether $X\in K_{m}$ or not, we obtain
\begin{equation}\label{eq:distance_estimate_1}
\d_{\rm BL}(\sigma^m,\sigma)\leq \frac{1}{m}+2\sigma(K_m^{\complement}).
\end{equation}

We now combine the entropy inequality with the bound $\log (1+x)\leq x$ to obtain, for $\theta>0$,
\begin{equation}\label{eq:bound_bad_probability}
\begin{split}
\sigma (K_m^{\complement}) & = \frac{1}{\theta} \E_{\sigma}[ \theta 1_{K_m^{\complement}} ] \leq \frac{1}{\theta}\Big( H(\sigma|\mu)+\log \bb{E}_{\mu} [e^{\theta 1_{K_m^{\complement}}}] \Big) \\
& = \frac{1}{\theta} \Big( \alpha + \log \big(1-\mu(K_m^{\complement})+e^\theta \mu(K_m^{\complement})\big) \Big) \\
& \leq \frac{\alpha}{\theta} + \frac{(e^{\theta}-1)}{\theta} \mu(K_m^{\complement}) \\
& \leq \frac{\alpha}{\theta} + \frac{(e^{\theta}-1)}{\theta}\frac{e^{-m^{2}-1}}{m},
\end{split}
\end{equation}
by the choice of $K_m$ in~\eqref{eq:choice_K}. Combining the equation above with~\eqref{eq:distance_estimate_1} concludes the proof of~\eqref{eq:distance_estimate}.

Choose now $\theta=m$ in~\eqref{eq:distance_estimate} to obtain
\begin{equation}
\d_{\rm BL}(\sigma^m,\sigma)\leq \frac{1}{m}+2\frac{\alpha}{m} + 2\frac{e^{-m^{2}+m-1}}{m^{2}} \leq \frac{3+2\alpha}{m},
\end{equation}
concluding the proof.
\end{proof}

Second, we provide a martingale that will be useful during the proof.
\begin{lemma}\label{lemma:martingale}
Assume either that $H(\upsilon|\mu)$ or $\sup_{m}H(\upsilon^{m}|\mu^{m})$ is finite. Then
\begin{equation}\label{eq:martingale}
S_m = \frac{\d \upsilon^m}{\d \mu^m} \circ \pi^{m}
\end{equation}
is an uniformly-integrable martingale in the probability space $\big(M, \mc{B}(M), \mu \big)$ with respect to the filtration $\big(\mathcal{F}_{m}\big)_{m \in \bb{N}}$.
\end{lemma}

\begin{proof}
Assume first that $H(\upsilon|\mu) < \infty$. In this case, $\tfrac{\d \upsilon}{\d \mu}$ exists and
\begin{equation}
\hat{S}_m = \E_\mu \Big[ \frac{\d \upsilon}{\d \mu} \Big| \F_m \Big]
\end{equation}
is a uniformly-integrable martingale. It follows directly from the definition of conditional expectation and Radon-Nikodyn derivative  that $\hat{S}_{m}=S_{m}$ almost surely for every $m \in \bb{N}$, concluding the proof of the first case.

Assume now that $\sup_{m}H(\upsilon^{m}|\mu^{m}) < \infty$ and observe that this implies that $S_m$ is well defined for all $m \in \bb{N}$, has expectation one, and is non-negative. We first have to verify that $\E[S_{m+1} |\F_m] = S_m$. Take an element $A_{m,k} \in \F_m$, with $0 \leq k \leq \ell_{m}$, so that
\begin{equation}
    \E_\mu [S_{m} \cdot 1_{A_{m,k}}] = \frac{\upsilon(A_{m,k})}{\mu(A_{m,k})} \E_{\mu}[ 1_{A_{m,k}}] = \upsilon(A_{m,k}).
\end{equation}
Now, let $B_1, \cdots, B_j$ elements of $\F_{m+1}$ such that $\cup_{i=1}^j B_i = A_{m,k}$, then
\begin{equation}
\begin{split}
    \E_\mu\left[\E_\mu[S_{m+1} | \F_m \right] \cdot 1_{A_{m,k}}] & = \E_\mu\left[ \E_\mu\left[S_{m+1} \left. \sum_{i=1}^j 1_{B_i} \right| \F_m \right] \right] \\
    & = \E_\mu\left[ \E_\mu\left[\left. \sum_{i=1}^j  \frac{\upsilon(B_i)}{\mu(B_i)} 1_{B_i} \right| \F_m \right] \right] \\
    &= \sum_{i=1}^j  \frac{\upsilon(B_i)}{\mu(B_i)} \mu(B_i) \\
    &= \upsilon(A_{m,k})=\E[S_m1_{A_{m,k}}],
\end{split}
\end{equation}
which implies $ \E[S_{m+1} |\F_m] = S_m$  , concluding our first statement.

In order to verify uniform integrability of $S_{n}$, observe that
\begin{equation}
\E_\mu[S_m \log S_m] = H(\upsilon^{m}|\mu^{m})  \leq \sup_{m}  H(\upsilon^{m}|\mu^{m}):=K < \infty.
\end{equation}

Now, for each $M>0$ we have, uniformly in $m \in \bb{N}$,
\begin{equation}
    E_\mu[S_m 1_{\{S_m \geq M\}}] \leq \E_\mu \left[ S_m 1_{\{S_m \geq M\}}  \dfrac{\log S_m}{\log M} \right] \leq \frac{K}{\log M}.
\end{equation}
Therefore, $S_m$ is a uniformly-integrable martingale, concluding the proof of the lemma.
\end{proof}

We are now in position to prove the first part of Lemma~\ref{le:supgoodpart}. 
\begin{proof}[Proof of Lemma~\ref{le:supgoodpart}]
We first observe that, via the variational definition of entropy and the fact that $M_{m} \subset M_{m+1}$, for all $m \in \bb{N}$, we obtain that $H(\upsilon^m|\mu^m)$ is monotone increasing in $m$ (following the steps pointed out in Remark~\ref{remark:rate_function} or directly as a consequence of~\cite[Corollary 5.2.2]{gray}). In particular,
\begin{equation}
\sup_m H(\upsilon^m|\mu^m)=\lim_{m \to \infty} H(\upsilon^m|\mu^m)
\end{equation}
and thus it suffices to verify that
\begin{equation}\label{eq:limit}
\lim_m H(\upsilon^m|\mu^m) = H(\upsilon|\mu).
\end{equation}

Again from the variational definition of relative entropy we have $ H(\upsilon^m|\mu^m) \leq H(\upsilon|\mu)$, for all $m \in \bb{N}$, so that
\begin{equation}
\limsup_{m \to \infty} H(\upsilon^{m}|\mu^{m}) \leq   H(\upsilon|\mu).
\end{equation}

We now work on the proof of the reverse inequality. The strategy of the proof is as follows. If at least one of the two quantities of interest is finite, we have access to the uniformly-integrable martingale $S_{m}$ given by the Radon-Hikodyin derivative of $\upsilon^{m}$ with respect to $\mu^{m}$. As we will see, this martingale converges in $L^{1}$ and almost surely to $\tfrac{\d \upsilon}{\d \mu}$, which will yield the result when combined with Fatou's Lemma.

Assume that either $H(\upsilon|\mu)<\infty$ or $\sup_{m}H(\upsilon^{m}|\mu^{m})<\infty$. The martingale $S_{m}$ introduced in~\eqref{eq:martingale} is uniformly integrable and thus converges almost surely and in $L^{1}$ to a random variable $X$.

In the case $H(\upsilon|\mu)<\infty$, we have
\begin{equation}
X=E_{\mu}\Big[ \frac{\d \upsilon}{\d \mu} \Big| \mathcal{F}_{\infty} \Big] = \frac{\d \upsilon}{\d \mu},
\end{equation}
since $\mathcal{F}_{\infty} = \mathcal{B}(M)$ (see Lemma~\ref{lemma:good_partition}).
If we are in the case $\sup_{m}H(\upsilon^{m}|\mu^{m})<\infty$, the above also holds precisely because elements in $\mathcal{B}(M)$ can be approxiamted by elements in $\cup_{m=1}^{\infty} \mathcal{F}_{m}$.

We now note that, since $x \log x \geq -e^{-1}$, Fatou's Lemma implies
\begin{equation}
\liminf_{m} \E_\mu \big[ S_{m} \log S_{m} \big] \geq \E_\mu \Big[ \frac{\d \upsilon}{\d \mu} \log \frac{\d \upsilon}{\d \mu} \Big],
\end{equation}
which verifies~\eqref{eq:limit_entropy} (see also Lemma~\ref{lemma:entropy}) and concludes the proof of the lemma.

\end{proof}

\section{Proof of Lemma~\ref{lemma:supinf}}\label{section:proof_supinf}
~
\par In this section we prove Lemma~\ref{lemma:supinf}. We fix $m_{0} \in \bb{N}$ and denote by
\begin{equation}\label{eq:sanov:defIstar}
I^0(\upsilon) := \sup_{m \geq m_{0}} \inf_{\sigma \in \bar{B}_{\frac{1}{\sqrt{m}}}(\upsilon)} H(\sigma|\mu^m). 
\end{equation}
Our goal is to show that $I^0(\upsilon)= H(\upsilon|\mu)$. We will prove this in two steps, by checking that $I^0(\upsilon) \leq H(\upsilon|\mu)$ and $I^0(\upsilon) \geq H(\upsilon|\mu)$. The first inequality is verified in the next paragraph. The reverse inequality is more delicate and we dedicate the rest of the section to verify it.

Let us check that $I^0(\upsilon) \leq H(\upsilon|\mu)$. Indeed, the inequality is trivial if $H(\upsilon|\mu) = \infty$. If on the other hand this entropy is finite, we have, in view of Lemma~\ref{le:supgoodpart},
\begin{equation}
\d_{\rm BL}(\upsilon, \upsilon^{m}) \leq \frac{c}{m } \leq \frac{1}{\sqrt{m}},
\end{equation}
for $m$ large enough, from which our claim follows by noting that $\upsilon^{m} \in \bar{B}_{\frac{1}{\sqrt{m}}}(\upsilon)$ and applying Lemma~\ref{le:supgoodpart}.

We now focus on the proof of the inequality
\begin{equation}\label{eq:goal}
I^0(\upsilon)\geq H(\upsilon|\mu).
\end{equation}
Once again we assume that $I^0(\upsilon)<\infty$, since the alternative case is trivial.

The first observation we make is that~\eqref{eq:goal} follows if, for any $\alpha > 0$,
\begin{equation} \label{implicationInequality}
I^0(\upsilon) < \alpha \text{ implies } H(\upsilon|\mu) \leq \alpha.
\end{equation}

This follows directly from the following lemma together with the lower semicontinuity of the relative entropy $H( \,\cdot\, |\mu)$.

\begin{lemma}
If $I^0(\upsilon) < \alpha$, then, for every $\varepsilon>0$, there exists $\rho \in B_{\varepsilon}(\upsilon)$ such that
\begin{equation}
H(\rho|\mu) \leq \alpha.
\end{equation}
\end{lemma}

\begin{proof}
Our goal will be find $\rho$ such that $H(\rho|\mu) \leq \alpha$ and $\d_{\rm BL}(\upsilon, \rho)< \varepsilon$. Fix $m \geq m_{0}$ large enough such that
\begin{equation}
\frac{3+2\alpha}{m}+\frac{1}{\sqrt{m}} < \varepsilon,
\end{equation}

Recall from~\eqref{eq:sanov:defIstar} that $I^0(\upsilon) < \alpha$ implies that there exists $\sigma \in \bar{B}_{\frac{1}{\sqrt{m}}}(\upsilon)$ such that $H(\sigma|\mu^{m}) \leq \alpha$. Notice that, since this entropy is finite, we have $\sigma = \sigma^{m}$.

Define
\begin{equation}\label{eq:sanov:defnu}
\rho(F):=\sum_{i=1}^{\ell_m}\frac{\sigma(A_{m, i})}{\mu (A_{m, i})} \mu(F \cap A_{m, i}) = \sum_{i=0}^{\ell_m} \mu(F | A_{m, i}) \sigma(A_{m, i}).
\end{equation}
Via direct substitution it follows that $\rho^{m} = \sigma$. We claim that $H(\rho|\mu)=H(\sigma|\mu^{m}) \leq \alpha$ and $\d_{\rm BL}(\upsilon, \rho)< \varepsilon$.

In order to verify that $H(\rho|\mu)=H(\sigma|\mu^{m})$, observe that 
\begin{equation}
  \begin{split}
  H(\rho^{m+j}|\mu^{m+j}) & =\sum_{i=0}^{\ell_{m+j}} \rho(A_{m+j, i}) \log \frac{\rho(A_{m+j, i})}{\mu(A_{m+j, i})}\\
   &=\sum_{i=0}^{\ell_m}\sum_{k:A_{m+j, k}\subset A_{m, i}} \rho(A_{m+j, k})\log\frac{ \rho(A_{m+j, k})}{\mu(A_{m+j, k})}.
  \end{split}
\end{equation}
Furthermore, if  $A_{m+j, k} \subset A_{m, i}$, then, from~\eqref{eq:sanov:defnu},
\begin{equation}
\rho(A_{k+j, m})=\frac{\sigma(A_{m, i})}{\mu(A_{m, i})}\mu(A_{m+j, k}).
\end{equation}
Therefore,
\begin{equation}
\begin{split}
    H(\rho^{m+j}|\mu^{m+j}) & = \sum_{i=0}^{\ell_m} \sum_{k:A_{m+j, k}\subset A_{m, i}} \frac{\sigma(A_{m, i})}{\mu(A_{m, i})}\mu(A_{m+j, k}) \log\frac{\sigma(A_{m, i})}{\mu(A_{m, i})} \\
    & = H(\sigma|\mu^{m}),
\end{split}
\end{equation}
since
\begin{equation}
\sum_{k:A_{m+j, k} \subset A_{m, i}} \mu(A_{m+j, k})= \mu(A_{m, i}).
\end{equation}
In particular, from Lemma~\ref{le:supgoodpart}, $H(\rho|\mu) = H(\sigma|\mu^{m}) \leq \alpha$.

Finally, we now prove that $\d_{\rm BL}(\rho, \upsilon) < \varepsilon$ by estimating
\begin{equation}
\begin{split}
\d_{\rm BL}(\rho, \upsilon) & \leq \d_{\rm BL}(\rho, \rho^{m}) + \d_{\rm BL}(\rho^{m}, \sigma) + \d_{\rm BL}(\sigma, \upsilon) \\
& \leq \frac{3+2\alpha}{m}+\frac{1}{\sqrt{m}} \leq \varepsilon,
\end{split}
\end{equation}
where the last line uses Lemma~\ref{lemma:convergenceNukDiv}, since $H(\rho|\mu)$ is bounded by $\alpha$ and recalling that $\rho^m=\sigma$. This concludes the proof.
\end{proof}

\bibliographystyle{plain} 

\bibliography{mybib}

\end{document}